\documentclass[12pt]{amsart}
\usepackage{amsfonts,amscd,latexsym,amsmath,amssymb,cancel,enumerate}
\theoremstyle{plain}
\newtheorem{master}{Master}[section]

\newtheorem{thm}[master]{Theorem}
\newtheorem{fact}[master]{Fact}
\newtheorem{lem}[master]{Lemma}

\newtheorem{claim}[master]{Claim}
\theoremstyle{definition}
\newtheorem{defin}[master]{Definition}

\theoremstyle{remark}
\newtheorem{remark}[master]{Remark}
\numberwithin{equation}{section}
\newcommand{\Ur}{\mathbb{U}}
\newcommand{\Rea}{\mathbb{R}}
\newcommand{\Nat}{\mathbb{N}}
\newcommand{\Rat}{\mathbb{Q}}
\newcommand{\Age}{\mathcal{K}}

\newcommand{\Hol}{\mathbb{H}}
\newcommand{\Gur}{\mathbb{G}}
\newcommand{\Span}{\mathrm{span}}
\newcommand{\norm}{\|\cdot\|}
\newcommand{\dist}{\mathrm{dist}}
\begin{document}
\title[Structures on the Urysohn and Gurarij spaces]{Universal and homogeneous structures on the Urysohn and Gurarij spaces}
\author{Michal Doucha}
\address{Institute of Mathematics\\ Polish Academy of Sciences\\
00-656 Warszawa, Poland}
\address{Laboratoire de Math\' ematiques de Besan\c con\\Universit\' e de Franche-Comt\' e\\France}
\email{m.doucha@post.cz}
\thanks{The author was supported by funds allocated to the implementation of the international co-funded project in the years 2014-2018, 3038/7.PR/2014/2, and by the EU grant PCOFUND-GA-2012-600415, and by the region Franche-Comt\' e, France.}
\begin{abstract}
Using Fra\" iss\' e theoretic methods we enrich the Urysohn universal space by universal and homogeneous closed relations, retractions, closed subsets of the product of the Urysohn space itself and some fixed compact metric space, $L$-Lipschitz map to a fixed Polish metric space. The latter lifts to a universal linear operator of norm $L$ on the Lispchitz-free space of the Urysohn space.

Moreover, we enrich the Gurarij space by a universal and homogeneous closed subspace and norm one projection onto a $1$-complemented subspace. We construct the Gurarij space by the classical Fra\" iss\' e theoretic approach.
\end{abstract}
\maketitle
\section*{Introduction}
The Urysohn space and the Gurarij space are two main examples of universal and homogeneous metric structures. The former, in the category of separable complete metric spaces, was constructed by P. Urysohn already in 1927 in \cite{Ury}, while the latter, in the category of separable Banach spaces, was constructed by Gurarij in 1966 in \cite{Gu}. The characterizing property of the Urysohn space $\Ur$ is that for any finite subset $M$ (including the empty subset) of $\Ur$, any `abstract' finite metric extension $M'$ of $M$ can be realized in $\Ur$. This property also implies that $\Ur$ contains an isometric copy of any separable metric space. The characterizing property of the Gurarij space is that for any finite dimensional subspace $X$ of $\Gur$, any `abstract' finite dimensional extension $X'$ can be realized with an `$\varepsilon$-accuracy' for any $\varepsilon>0$. This property also implies that $\Gur$ contains a linearly isometric copy of any separable normed vector space.

The aim of this paper is to enrich the Urysohn and Gurarij spaces by some additional structure so that they remain universal and homogeneous with that added structure. There are several motivations for proving results of this type. First, by constructing an object (on the Urysohn or Gurarij space) of a certain type which is universal and homogeneous, one gets a better understanding of the whole class of all objects of such a type, especially how they are approximated by objects from the simpler subclass of finitely generated ones.

Secondly, these results provide a general way of coding metric structures. The common theme in modern descriptive set theory is to consider a class of certain countable or separable structures and decide the complexity of various equivalence relations on this class, e.g. the relation of isometry on the class of Polish metric spaces, the relation of isomorphism on the class of separable Banach spaces, etc. The starting point is to represent such a class as a Polish, resp. standard Borel, space. One approach is to take a universal object of the class and then to consider a `space of all its subobjects'. In this way Gao and Kechris (in \cite{GaoKe}) initiated the investigation of the isometry relation on Polish metric spaces. Their standard Borel space of all Polish metric space was the Effros-Borel space (see \cite{Ke} as a reference for notions of descriptive set theory) of the closed subsets of the Urysohn space. Analogously, Ferenczi, Louveau and Rosendal (in \cite{FLR}) studied the isomorphism relation on 
separable Banach spaces by taking the set of all closed subspaces of $C([0,1])$, which is a Borel subset of the Effros-Borel space of closed subsets of $C([0,1])$. We provide new universal objects whose Effros-Borel spaces of closed subsets can serve in the same way. Especially with our first result, mentioned below, on universal closed relations on the Urysohn space we had in mind this purpose. The Urysohn space with finitely many universal closed relation of an arbitrary arity, resp. its Effros-Borel space of closed subsets, can be used to code general Polish metric structures.

To mention some results of that kind, we refer the reader to \cite{GarKub} and \cite{CGK}, where the Gurarij space, resp. the $p$-Gurarij space, were enriched by a universal and homogeneous linear operator. We also mention the author's constructions of universal metric groups in \cite{Do} and \cite{Do2} which can also be viewed as an enriching the Urysohn space by group structures.\\

We roughly summarize the main results below. The precise statements are in the appropriate sections.
\begin{thm}
The Urysohn universal metric space can be enriched:
\begin{itemize}
\item with finitely many universal and homogeneous closed relations of an arbitrary arity,
\item with a universal and homogeneous $1$-Lipschitz retraction onto a universal and homogeneous retract subspace,
\item with a universal and homogeneous closed subset of the product of itself (the Urysohn space) and an arbitrary fixed compact metric space,
\item with a universal and homogeneous $L$-Lipschitz function to any fixed Polish metric space, for any fixed $L>0$; moreover, when the fixed Polish metric space is in fact a Banach space, then the $L$-Lipschitz functions lifts to a universal linear operator of norm $L$ from the Holmes space to that fixed Banach space.

\end{itemize}
\end{thm}
\begin{thm}
The Gurarij universal Banach space can be enriched:
\begin{itemize}
\item with a universal and homogeneous closed subspace,
\item with a universal and homogeneous norm one projection onto a universal and homogeneous $1$-complemented subspace.

\end{itemize}
\end{thm}
\section{Preliminaries}
We expect the reader to be familiar with basic Fra\" iss\' e theoretic constructions. At least knowing the standard Fra\" iss\' e theoretic construction of the Urysohn space is important for understanding our construction (as opposed to the other popular construction due to Kat\v etov in \cite{Kat}). We refer the reader to Chapter 5 of the book \cite{Pe} which is devoted to the Urysohn space and contains a construction of this space which is in the same spirit as our constructions here. For a general exposition of Fra\" iss\' e theory, we refer the reader to Chapter 7 of \cite{Ho} which is devoted to this topic, and then to \cite{Kub} that contains a general category-theoretic approach to Fra\" iss\' e theory.\\

We make a brief overview of Fra\" iss\' e theory here. Let $\Age$ be a class of (finitely generated) structures of some type and moreover, let $\mathcal{E}$ be some class of distinguished embeddings between structures of $\Age$. We say that $(\Age,\mathcal{E})$
\begin{itemize}
\item is countable if it contains only countably many structures up to isomorphism from $\mathcal{E}$,
\item has the joint-embedding property if for every $A,B\in \Age$ there is $C\in \Age$ and embeddings of $A$, resp. $B$, into $C$ that belong to $\mathcal{E}$,
\item has the amalgamation property if for every $A,B,C$ and embeddings $\iota_B:A\hookrightarrow B$ and $\iota_C:A\hookrightarrow C$ from $\mathcal{E}$ there exist $D\in \Age$ and embeddings $\rho_B:B\hookrightarrow D$ and $\rho_C:C\hookrightarrow D$ from $\mathcal{E}$ such that $\rho_C\circ\iota_C=\rho_B\circ \iota_B$.
\end{itemize}
If $(\Age,\mathcal{E})$ satisfies all these conditions then we call it a Fra\" iss\' e class. The following is the classical Fra\" iss\' e theorem. Let us note that when $A\in \Age$ and $K$ is a direct limit $K_1\to K_2\to\ldots$ of structures from $\Age$, then by saying that an embedding $\iota:A\hookrightarrow K$ belongs to $\mathcal{E}$ we mean that there exists $n$ so that $\iota\in \mathcal{E}$ in fact goes from $A$ to $K_n$
\begin{thm}[Fra\" iss\' e theorem]\label{thmFraisse}
Let $(\Age,\mathcal{E})$ be a Fra\" iss\' e class. Then there exists a limit structure $K$, called the Fra\" iss\' e limit, which is a direct limit of the form $K_1\to K_2\to\ldots$, where $K_i\in \Age$, for every $i$, and the embedding between $K_i$ and $K_{i+1}$ is in $\mathcal{E}$, for every $i$. The following properties characterize it up to isomorphism (among countable, resp. separable, structures of the same type):
\begin{itemize}
\item for every $A\in \Age$ there exists an embedding $\iota_A:A\hookrightarrow K$ from $\mathcal{E}$,
\item for every $A,B\in \Age$ and embeddings $\iota_A:A\hookrightarrow K$ and $\rho: A\hookrightarrow B$ from $\mathcal{E}$ there exists an embedding $\iota_B:B\hookrightarrow K$ from $\mathcal{E}$ such that $\iota_B\circ \rho=\iota_A$.

\end{itemize}
\end{thm}
One can derive two additional properties of $K$ from those two stated above:
\begin{enumerate}
\item If $L$ is a structure that is a direct limit $L_1\to L_2\to\ldots$, where for each $i$, $L_i\in \Age$, and the embedding from $L_i$ to $L_{i+1}$ is from $\mathcal{E}$, then there exists an embedding of $L$ into $K$.
\item If $A,B\in \Age$ are isomorphic and embedded into $K$ via $\iota_A:A\hookrightarrow K$ and $\iota_B:B\hookrightarrow K$ from $\mathcal{E}$, then there exists an automorphism of $K$ that sends $\iota_A[A]$ onto $\iota_B[B]$.

\end{enumerate}
The first property is called the \emph{universality} of $K$; the latter is called the \emph{homogeneity}, or sometimes ultrahomogeneity, of $K$.\\

The best example for us is the following. Let $\Age$ be the class of all finite metric spaces with rational distances and consider the class $\mathcal{E}$ of all isometric embeddings between them. Then $(\Age,\mathcal{E})$ is a Fra\" iss\' e class with the limit being the \emph{rational Urysohn space} $\Rat \Ur$. We get from Theorem \ref{thmFraisse} the following characterization of $\Rat\Ur$.
\begin{fact}\label{charUry}
$\Rat\Ur$ is the unique countable rational metric space with the property that for every finite subspace $A$ and its finite extension, which is still rational, this extension of $A$ is actually realized within $\Rat\Ur$.
\end{fact}
The property from Fact \ref{charUry} is called the rational finite-extension property, or in the case when this extension is just by one point, the \emph{rational one-point extension property}. Obviously, one-point extension property implies finite-extension property.

The metric completion of $\Rat\Ur$ is the Urysohn space $\Ur$. We again refer to Chapter 5 of \cite{Pe} where this is proved.

In the section on the Gurarij space, we shall present a new construction of the Gurarij space that is Fra\" iss\' e theoretic in this classical sense. That will help us use similar ideas and techniques that we use for enriching the Urysohn space  to enrich the Gurarij space by an additional structure as well.\\

Since we will amalgamate metric structures often in the sequel, we make a brief description of that procedure here. 
\begin{defin}[(Greatest) metric amalgamation]\label{definamal}
Suppose we are given metric spaces $X_0,X_1,X_2$ such that $X_0$ is a subspace of both $X_1$ and $X_2$. The underlying set of the amalgam $X_3$ will be the disjoint union $X_0\coprod (X_1\setminus X_0)\coprod (X_2\setminus X_0)$. We need to define the distance between $x$ and $y$, where $x\in X_1\setminus X_0$ and $y\in X_2\setminus X_0$. In order to satisfy the triangle inequalities, we need to have $$\sup_{z\in X_0} |d_{X_1}(x,z)-d_{X_2}(z,y)|\leq d(x,y)\leq \inf_{z\in X_0} d_{X_1}(x,z)+d_{X_2}(z,y).$$ If we take for all such pairs the greatest extreme, then it is straightforward to check that this defines a metric on $X_3$ extending those on $X_1$ and $X_2$, and that such a metric is the greatest possible. If $X_1\setminus X_0=\{x\}$ and $X_2\setminus X_0=\{y\}$ then we can of course define the distance between $x$ and $y$ to be any number in between these two extremes.
\end{defin}

Let us also note that whenever we consider some metric on a product of metric spaces we mean the \emph{sum metric}.
\section{The Urysohn space}
In this section, we prove the results concerning the Urysohn space. Theorems are stated rather informally as it is convenient for readers familiar with the Fra\" iss\' e theory. In the remarks below the statements of theorems, we explain more all the technical details.
\begin{thm}\label{thm1}
Let $n_1\leq \ldots \leq n_m$ be an arbitrary non-decreasing sequence of natural numbers. Then there exist closed relations (subsets) $F_{n_i}\subseteq \Ur^{n_i}$, for $i\leq m$, such that the structure $(\Ur,F_{n_1},\ldots,F_{n_m})$ is universal and ultrahomogeneous, and it is unique (up to an isometry preserving the relations) with this property.
\end{thm}
\begin{remark}
Let $n_1\leq \ldots \leq n_m$ be like in the statement of the theorem. Then $(\Ur,F_{n_1},\ldots,F_{n_m})$ is uniquely characterized by the following property: let $A_1,A_2\subseteq (\Ur,F_{n_1},\ldots,F_{n_m})$ be two finite isomorphic substructures; i.e. viewing $A_,A_2$ as finite subsets of $\Ur$ there is an isometry $\iota:A_1\rightarrow A_2$ such that for each $i\leq m$ and each $\vec{x}\in A_1^{n_i}$ we have $\dist(\vec{x},F_{n_i})=\dist(\iota(\vec{x}),F_{n_i})$. Then $\iota$ extends to an automorphism of $(\Ur,F_{n_1},\ldots,F_{n_m})$, i.e. an autoisometry of $\Ur$ that preserves distances to the closed sets $F_{n_i}$, for every $i\leq m$.

As a consequence, we have the following universality of\\ $(\Ur,F_{n_1},\ldots,F_{n_m})$. Let $X$ be any Polish metric space equipped with closed relations $G_{n_1},\ldots,G_{n_m}$ of the same arity, i.e. for each $i\leq m$, $G_{n_i}$ is a closed subset of $X^{n_i}$. Then there exists an isometric embedding $\iota:X\hookrightarrow \Ur$ such that for each $i\leq m$ and $\vec{x}\in X^{n_i}$ we have $\dist(\vec{x},G_{n_i})=\dist(\iota(\vec{x},F_{n_i})$.
\end{remark}
\begin{proof}
Consider the countable class $\Age_1$ of finite rational metric structures defined as follows. We have $A\in \Age_1$ if the following is satisfied:
\begin{itemize}
\item $A$ is a finite rational metric space,
\item for each $i\leq m$ there is a $1$-Lipschitz rational function $p_i: A^{n_i}\rightarrow \Rat^+_0$ which interprets as a distance function from the desired closed set $F_{n_i}$.

\end{itemize}
The class of embeddings consists of \emph{all} isometric embeddings between elements of $\Age_1$ that preserve the functions $p_i$. We shall thus write just $\Age_1$ instead of full $(\Age_1,\mathcal{E}_1)$.

To check that $\Age_1$ is a Fra\" iss\' e class we have to verify that it is countable and has the joint and amalgamation properties. Since all the functions take values in rationals, it is countable.

We check the amalgamation property. The joint embedding property is similar, just easier. Let $A,B,C\in \Age_1$ be structures such that $A$ is a common substructure of both $B$ and $C$. The underlying metric space of the amalgam $D$ will be the greatest metric amalgamation of $B$ and $C$ over $A$ as defined in Definition \ref{definamal}. For each $i\leq m$ we have to extend $p_{n_i}$ from $B^{n_i}\cup C^{n_i}$ onto $D^{n_i}$. We take the greatest $1$-Lipschitz extension of $p_{n_i}$; i.e. for any $\vec{d}\in D^{n_i}$ we set $$p_{n_i}(\vec{d})=\min \{p_{n_i}(\vec{a})+d(\vec{d},\vec{a}):\vec{a}\in B^{n_i}\cup C^{n_i}\}.$$ That is a standard way of extension of real valued (resp. in this case, rational valued) Lipschitz functions, so we leave to the reader to check that it works.

It follows that $\Age_1$ has a Fra\" iss\' e limit. The limit is a countable rational metric space $M$ together with rational $1$-Lipschitz functions $\Rat P_{n_i}:M\rightarrow \Rat^+_0$, for every $i\leq m$. However, the metric space $M$ is actually the rational Urysohn space $\Rat \Ur$. That immediately follows from Fact \ref{charUry} by checking that $M$ has the metric rational one-point extension property. Denote also by $\Rat F_{n_i}$, for each $i\leq m$, the set $\{x\in \Rat\Ur^{n_i}:\Rat P_{n_i}(x)=0\}$. It follows from the Fra\" iss\' e theorem that the metric structure $(\Rat\Ur, \Rat P_{n_1},\ldots,\Rat P_{n_m})$ has the following universality and \emph{rational one-point extension} properties:
\begin{itemize}
\item for every $A\in \Age_1$ there exists an embedding\\ $\iota:A\hookrightarrow (\Rat\Ur, \Rat P_{n_1},\ldots)$, i.e. for each $i\leq m$ and for every $\vec{a}\in A^{n_i}$ we have $p_{n_i}(\vec{a})=\Rat P(\iota(\vec{a}))$,
\item for every $A\in \Age_1$, every embedding $\iota:A\hookrightarrow (\Rat\Ur,\Rat P_{n_1},\ldots)$, and every one-point extension $A\subseteq B\in \Age_1$ there exists an extension $\iota\subseteq \tilde{\iota}:B\hookrightarrow (\Rat\Ur,\Rat P_{n_1},\ldots)$.

\end{itemize}
We now take the completion of this rational structure to obtain\\ $(\Ur, P_{n_1},\ldots,P_{n_m})$, where $\Ur$ is the Urysohn space, the completion of $\Rat\Ur$, and for each $i\leq m$, $P_{n_i}$ is the unique extension of $\Rat P_{n_i}$ onto $\Ur$. Denote also by $F_{n_i}$ the set $\overline{\Rat F_{n_i}}=\{x\in \Ur^{n_i}:P_{n_i}(x)=0\}$, i.e. the closure of $\Rat F_{n_i}$ in $\Ur^{n_i}$.

Let $\overline{\Age_1}$ denote the `real' version of $\Age_1$, i.e. finite structures of the same type as those in $\Age_1$ with the difference that we allow all functions there, including the metric, to take values in all the reals, not only in the rationals. In order to prove Theorem \ref{thm1} we show that $(\Ur, P_{n_1},\ldots,P_{n_m})$ has the same universality and one-point extension properties with respect to $\overline{\Age_1}$ as $(\Rat\Ur,\Rat P_{n_1},\ldots)$ does with respect to $\Age_1$. That is the content of the next lemma.
\begin{lem}\label{thm1lem1}
For every $A\in \overline{\Age_1}\cup\{\emptyset\}$, every embedding $\iota:A\hookrightarrow (\Ur,P_{n_1},\ldots)$, and every one-point extension $A\subseteq B\in \overline{\Age_1}$ there exists an extension $\iota\subseteq \tilde{\iota}:B\hookrightarrow (\Ur, P_{n_1},\ldots)$.
\end{lem}
By allowing $A$ in the lemma to be empty, we get also the `real version' of universality property, i.e. that for every $A\in \overline{\Age_1}$ there is an embedding $\iota:A\hookrightarrow (\Ur, P_{n_1},\ldots)$. Once Lemma \ref{thm1lem1} is proved, we are done. Indeed, the universality, ultrahomogeneity and uniqueness of $(\Ur, P_{n_1},\ldots)$ follow from it by the similar arguments for proving universality, ultrahomogeneity and uniqueness of the Urysohn space using the one-point extension property. For a reader not familiar with such arguments, we provide a sketch of the proof of universality.

Let $(X,G_{n_1},\ldots,G_{n_m})$ be a Polish metric structure where $X$ is a Polish metric space and for each $i\leq m$, $G_{n_i}$ is a closed subset of $X^{n_i}$. Denote also by $Q_{n_i}:X^{n_i}\rightarrow \Rea^+_0$ the distance function from $G_{n_i}$. Let $D=\{d_n:n\in \Nat\}\subseteq X$ be some countable dense subset. Using Lemma \ref{thm1lem1} inductively, we build an increasing sequence of embeddings $\iota_1\subseteq \iota_2\subseteq\ldots$, where for each $j$, $\iota_j:\{d_1,\ldots,d_j\}\hookrightarrow (\Ur, P_{n_1},\ldots)$ is isometric, and such that for each $i\leq m$ and every $\vec{d}\in \{d_1,\ldots,d_j\}^{n_i}$ we have $Q_{n_i}(\vec{d})=P_{n_i}(\iota_j(\vec{d}))$. Then we take $\iota=\bigcup_j \iota_j:D\rightarrow (\Ur, P_{n_1},\ldots)$. Since $\iota$ is an isometry, we can extend it to $\bar{\iota}:X\rightarrow (\Ur, P_{n_1},\ldots)$. Since for each $i\leq m$, $Q_{n_i}$ is $1$-Lipschitz, we also get that for every $\vec{x}\in X^{n_i}$ we have that $Q_{n_i}(\vec{x})=P_{n_i}(\bar{\iota}(\vec{x})
)
$. In particular, for each $\vec{x}\in X^{n_i}$ we have that $\vec{x}\in G_{n_i}$ if and only if $\bar{\iota}(\vec{x})\in F_{n_i}$, and we are done. Homogeneity and uniqueness are done similarly. We refer the reader again to Chapter 5 in \cite{Pe} where these facts are proved for the plain Urysohn space.\\

Thus we are left to prove Lemma \ref{thm1lem1}. We will do it in a series of three claims. We need one definition before.
\begin{defin}\label{thm1def}
We say that a class $K\subseteq \overline{\Age_1}$ has an almost-one-point extension property if for every $A\in K$, every one-point extension $A\subseteq B=A\coprod \{b\}\in \overline{\Age_1}$ and every $\varepsilon>0$ there exists an extension $B'=A\coprod \{b'\}\in K$ such that for every $a\in A$ we have $|d(a,b)-d(a,b')|<\varepsilon$, and for every $i\leq m$ and $\vec{a}\in B^{n_i}$ we have $|p_{n_i}(\vec{a})-p_{n_i}(\vec{a'})|<\varepsilon\cdot n_i$, where $\vec{a'}$ is obtained from $\vec{a}$ by replacing each occurrence of $b$ by $b'$.

Analogously, we say that a substructure $(U,P_{n_1}\upharpoonright U^{n_1},\ldots,P_{n_m}\upharpoonright U^{n_m})$ has an almost-one-point extension property if for every $A\in \overline{\Age_1}$, every embedding $\iota:A\hookrightarrow (U,P_{n_1}\upharpoonright U^{n_1},\ldots)$, every $\varepsilon>0$ and every one-point extension $A\subseteq B=A\coprod \{b\}\in \overline{\Age_1}$ there exists an extension $B'=\iota[A]\coprod \{b'\}\subseteq (U,P_{n_1}\upharpoonright U^{n_1},\ldots)$ such that for every $a\in A$ we have $|d(a,b)-d(\iota(a),b')|<\varepsilon$, and for every $i\leq m$ and $\vec{a}\in B^{n_i}$ we have $|p_{n_i}(\vec{a})-p_{n_i}(\vec{a'})|<\varepsilon\cdot n_i$, where $\vec{a'}$ is obtained from $\vec{a}$ by replacing each occurrence of $b$ by $b'$ and each occurrence of $a_j$ by $\iota(a_j)$, for every $j\leq n$.
\end{defin}
\begin{claim}\label{thm1claim1}
$\Age_1$ has the almost-one-point extension property.
\end{claim}
\begin{claim}\label{thm1claim2}
If a (not necessarily complete) substructure\\ $(U,P_{n_1}\upharpoonright U^{n_i},\ldots,P_{n_m}\upharpoonright U^{n_m})\subseteq (\Ur, P_{n_1},\ldots)$ has an almost-one-point extension property, then so does its completion.
\end{claim}
\begin{claim}\label{thm1claim3}
If a complete substructure $(U,P_{n_1}\upharpoonright U^{n_i},\ldots,P_{n_m}\upharpoonright U^{n_m})\subseteq (\Ur, P_{n_1},\ldots)$ has an almost-one-point extension property, then it has the one-point extension property.
\end{claim}
Lemma \ref{thm1lem1} follows from these claims. Note at first, that the statement that $\Age_1$ has an almost-one-point extension property is equivalent with the statement that $(\Rat \Ur,\Rat P_{n_1},\ldots)$ has an almost-one-point extension property. By Claim \ref{thm1claim1}, $(\Rat \Ur,\Rat P_{n_1},\ldots)$ has an almost-one-point extension property. Then by Claim \ref{thm1claim2}, its completion, $(\Ur, P_{n_1},\ldots)$ has also an almost-one-point extension property and using Claim \ref{thm1claim3} it must thus have the one-point extension property.\\

Let us now prove Claim \ref{thm1claim1}. Let $A\in \Age_1$ and $\varepsilon>0$ and let $A\subseteq B\in \overline{\Age_1}$ be some one-point extension. Let $f:A\rightarrow \Rea^+$ be the function $d(\cdot,b)$. Enumerate $A$ as $a_1,\ldots,a_n$ so that we have $f(a_1)\geq f(a_2)\geq\ldots\geq f(a_n)$. For each $j\leq n$ let $q_j$ be an arbitrary rational number in the interval $\left(f(a_j)+(j-1)\cdot \varepsilon/n,f(a_j)+j\cdot\varepsilon/n\right]$. Let $B'=A\coprod \{b'\}$ be a one-point extension of $A$ such that for every $j\leq n$ we have $d(a_j,b')=q_j$. The triangle inequalities are satisfied. Indeed, for every $j<k\leq n$ we have $$|d(a_j,b')-d(a_k,b')|=|q_j-q_k|=q_j-q_k\leq f(a_j)-f(a_k)\leq d(a_j,a_k)\leq$$ $$f(a_j)+f(a_k)\leq q_j+q_k=d(a_j,b')+d(a_k,b').$$
For each $i\leq m$ we also have to define $p_{n_i}$ on $(B')^{n_i}\setminus A^{n_i}$. Note that for every $a\in A$ we have $d(a,b)<d(a,b')$, thus let $\delta=\min\{d(a,b')-d(a,b):a\in A\}$. For every $\vec{a}\in B^{n_i}\setminus A^{n_i}$ let $\vec{a'}$ be the corresponding tuple from $(B')^{n_i}\setminus A^{n_i}$, i.e. where each occurrence of $b$ is replaced by $b'$. For each $\vec{a}\in B^{n_i}\setminus A^{n_i}$ we set $p_{n_i}(\vec{a'})$ to be an arbitrary rational in the interval $[p_{n_i}(\vec{a}),p_{n_i}(\vec{a})+\delta/2]$. Then for any $\vec{a},\vec{b}\in B^{n_i}\setminus A^{n_i}$ we have $$|p_{n_i}(\vec{a'})-p_{n_i}(\vec{b'})|\leq |p_{n_i}(\vec{a})-p_{n_i}(\vec{b})|+\delta\leq d(\vec{a},\vec{b})+\delta\leq d(\vec{a'},\vec{b'})$$ verifying that $p_{n_i}$ is $1$-Lipschitz. It follows that $B'\in \Age_1$ and it is as desired. That finishes the proof of Claim \ref{thm1claim1}.

Proof of Claim \ref{thm1claim2} is the same as the proof of Lemma 5.1.15 in \cite{Pe} and proof of Claim \ref{thm1claim3} is the same as the proof of Lemma 5.1.16 in \cite{Pe}, thus we refer the reader there.
\end{proof}
\begin{thm}\label{thm2}
There exists a universal and ultrahomogeneous $1$-Lipschitz retraction $R:\Ur\rightarrow F_\Ur\subsetneq \Ur$, where $F_\Ur$ is again isometric to $\Ur$.
\end{thm}
\begin{remark}
$(\Ur,R,F_\Ur)$ is uniquely characterized by the following property: let $A_1,A_2\subseteq \Ur$ be two finite subsets that are isomorphic, i.e. there exists an isometry $\iota:A_1\rightarrow A_2$ such that
\begin{itemize}
\item for every $x\in A_1$ we have $\iota\circ R(x)=R\circ\iota(x)$,
\item for every $x\in A_1$ we have $\dist(x,F_\Ur)=\dist(\iota(x),F_\Ur)$.
\end{itemize}
Then $\iota$ extends to an autoisometry of $\Ur$ still commuting with the retraction $R$ and preserving the distance from $F_\Ur$.

As a consequence we get the following universality property of $(\Ur,R,F_\Ur)$. Let $X$ be any Polish metric space equipped with a $1$-Lipschitz retraction $Q:X\rightarrow F_X$. Then there exists an isometric embedding $\iota:X\rightarrow \Ur$ such that for any $x\in X$ we have $R\circ\iota(x)=\iota\circ Q(x)$, and $\dist(x,F_X)=\dist(\iota(x),F_\Ur)$.
\end{remark}
\begin{proof}
Here we consider the following countable class $\Age_2$ of finite rational metric structures. A finite structure $A$ belongs to $\Age_2$ if
\begin{itemize}
\item $A$ is a rational metric space,
\item there is a rational $1$-Lipschitz function $p:A\rightarrow \Rat^+_0$ that interprets as a distance function from the desired universal retract,
\item there is a $1$-Lipschitz retraction $r:A\rightarrow A_F$, where $A_F=\{a\in A: p(A)=0\}$.

\end{itemize}
We again consider all the embeddings that preserve metric (i.e. they are isometric) and the functions $p$ and $r$.

It is clear that $\Age_2$ is countable. We again show just the amalgamation property. Suppose we have structures $A,B,C\in \Age_2$, where we assume that $A$ is a common substructure of both $B$ and $C$. As in the proof of Theorem \ref{thm1}, we take the greatest metric amalgam $D$ of $B$ and $C$ over $A$ and show that it works. All we need to do is to check that $p$ and $r$ on $D$ are still $1$-Lipschitz which is analogous to the proof that functions $p_{n_i}$'s on amalgams remain $1$-Lipschitz from the proof of the previous theorem.

Thus we get some a Fra\" iss\' e limit $(\Rat \Ur,\Rat R, \Rat P)$, where again $\Rat \Ur$ is the limit as a metric space which is again the rational Urysohn space. $\Rat R$ is the limit of the retractions and $\Rat P$ is the limit of the distance function $p$. By $F_{\Rat\Ur}\subseteq \Rat\Ur$ we denote the set $\{a\in \Rat\Ur:\Rat P(a)=0\}=\{a\in \Rat\Ur: \Rat R(a)=a\}$, i.e. the universal retract.

From the Fra\" iss\' e theorem we have the following universality and rational one-point extension property of $(\Rat \Ur,\Rat R,\Rat P)$:
\begin{itemize}
\item for every $(A,r,p)\in \Age_2$ we have an embedding\\ $\iota: A\hookrightarrow (\Rat \Ur,\Rat R,\Rat P)$, i.e. for every $a\in A$ we have $\Rat R\circ \iota(a)=\iota \circ r(a)$ and $\Rat P\circ \iota(a)=p(a)$,
\item for every $A\in \Age_2$, every embedding $\iota:A\hookrightarrow (\Rat \Ur,\Rat R,\Rat P)$, and every one-point extension $A\subseteq B\in \Age_2$ there exists an extension $\iota\subseteq \tilde{\iota}:B\hookrightarrow (\Rat \Ur,\Rat R,\Rat P)$

\end{itemize}
We take the completion to obtain the Urysohn space together with the retraction $R$ from $\Ur$ onto $F_\Ur=\overline{F_{\Rat \Ur}}$, which is the unique extension of $\Rat R$, and with the distance function $P$, which is the unique extension of $\Rat P$. The following lemma will finish the proof of the theorem.
\begin{lem}\label{thmlem2}
$(\Ur, R, P)$ has a one-point extension property, i.e. for every finite subset $A\subseteq \Ur$ and for every abstract one-point extension $B=A\coprod \{b\}$, there is a corresponding `concrete' one-point extension in $(\Ur, R, P)$.
\end{lem}
Let us formulate the almost-one-point extension property in this situation. $\overline{\Age_2}$ denotes the real finite substructures of the same type as those in $\Age_2$. That is an equivalent definition to $\overline{\Age_1}$ in the proof of the previous theorem.

We say that a class $K\subseteq \overline{\Age_2}$ has an almost-one-point extension property if for every $A\in K$, every extension $A\subseteq B=A\coprod \{b\}\in \overline{\Age_2}$ and every $\varepsilon>0$ there exists a (one or two point) extension $B'=A\coprod \{b',R(b')\}\in K$ such that for every $a\in A$ we have $|d(a,b)-d(a,b')|<\varepsilon$, $|p(b)-p(b')|<\varepsilon$, and if $R(b)\neq b$, i.e. $R(b)\in A$, then we have $d(R(b'),R(b))<\varepsilon$. Analogous definition is used for substructures of $(\Ur, R, F_\Ur)$ (as in Definition \ref{thm1def}).

It suffices to prove the following series of claims as in the proof of Theorem \ref{thm1}.
\begin{claim}\label{thm2claim1}
$\Age_2$ has the almost-one-point extension property.
\end{claim}
\begin{claim}\label{thm2claim2}
If a (not necessarily complete) substructure $(U,R\upharpoonright U,P\upharpoonright U)\subseteq (\Ur, R, P)$ has an almost-one-point extension property, then so does its completion.
\end{claim}
\begin{claim}\label{thm2claim3}
If a complete substructure $(U,R\upharpoonright U, P\upharpoonright U)\subseteq (\Ur, R, P)$ has an almost-one-point extension property, then it has the one-point extension property.
\end{claim}
Claim \ref{thm2claim1}, resp. Claim \ref{thm2claim3}, is proved as Claim \ref{thm1claim1}, resp. Claim \ref{thm1claim3}. We only prove Claim \ref{thm2claim2}. Let $(U,R\upharpoonright U,P\upharpoonright U)\subseteq (\Ur, R, P)$ be some not-complete substructure having an almost-one-point extension property and let $(\bar{U},R\upharpoonright \bar{U},P\upharpoonright \bar{U})$ be its completion. We note here that since $U$ is a substructure, it is closed under $R$, i.e. for every $u\in U$ we have $R(u)\in U$. We show that $(U,R\upharpoonright U,P\upharpoonright U)$ has the almost-one-point extension property. Let $A$ be some finite substructure of $(\bar{U},R\upharpoonright \bar{U},P\upharpoonright \bar{U})$ and let $A\subseteq B=A\coprod\{b\}\in \overline{\Age_2}$ be its one-point extension. Let $\varepsilon>0$ be given. Since $U$ is dense in $\bar{U}$ we can find for each $a\in A$ an element $u(a)\in U$ such that $d(u(a),a)<\varepsilon/2$ and so that for $a_1\neq a_2$, $u(a_1)\neq u(a_2)$. We set $A'=\{u(a):
a\in A\}\cup \{R(u(a)):a\in A\}$, i.e. closing $\{u(a):a\in A\}$ under $R$. Note that it is possible that $|A'|>|A|$ since there might be $a\neq b\in A$ such that $R(a)=R(b)$, however $R(u(a))\neq R(u(b))$. Nevertheless, since $R$ is $1$-Lipschitz we have that $d(R(a),R(u(a)))<\varepsilon/2$ for every $a\in A$. Now take the (greatest) metric amalgamation of $A'\cup A$ with $B$ over $A$. That gives some metric on $A'\cup\{b\}$. If $R(b)=b$ in $B$ then we define $B'=A'\cup\{b\}$ to be the extension of $A'$ where also $R(b)=b$; that gives some structure from $\overline{\Age_2}$. Otherwise, $R(b)=a$ for some $a\in A\subseteq B$ in $B$. Then we define $B'=A'\cup\{b\}$ to be the extension of $A'$ where $R(b)=u(a)$; that also gives some structure from $\overline{\Age_2}$.

We may thus find some almost-one-point extension for $\varepsilon/2$ in $U$. It is straightforward to check that it will be an almost-one-point extension for $A$ as well.
\end{proof}
\begin{thm}\label{thm3}
Let $K$ be an arbitrary compact metric space. Then there exists a closed subset $C\subseteq \Ur\times K$ that is universal and ultrahomogeneous.
\end{thm}
\begin{remark}
$(\Ur,C)$ is uniquely characterized by the following property: let $A_1,A_2\subseteq \Ur$ be two finite subsets of $\Ur$, and $K_0\subseteq K$ a finite subset of the compact metric space $K$ with the properties that
\begin{itemize}
\item there is an isometry $\iota:A_1\rightarrow A_2$,
\item for any $x\in A_1$ and $k\in K_0$ we have\\ $\dist((x,k),C)=\dist((\iota(x),k),C)$.
\end{itemize}
Then $\iota$ extends to an autoisometry $\bar{\iota}$ of $\Ur$ such that for any $x\in \Ur$ and $k\in K$ we have $\dist((x,k),C)=\dist((\bar{\iota}(x),k),C)$.

As a consequence we get that for any Polish metric space $X$ equipped with a closed subset $E\subseteq X\times K$ there is an isometric embedding $\iota:X\hookrightarrow \Ur$ such that for any $x\in X$ and $k\in K$ we have $\dist((x,k),E)=\dist((\iota(x),k),C)$.
\end{remark}
\begin{proof}
Let $D_K=\{q_n:n\in \Nat\}$ be an enumeration of some countable dense subset of $K$. For some metric space $M$ we want to consider a function $f:M\times\Nat\rightarrow \Rea^+_0$, where $f(x,n)$, for any $x\in M$ and $n\in \Nat$, is to be interpreted as the distance (again in the sum metric) of $(x,q_n)$ from some closed subset of $M\times K$. However, even if we restrict to finite rational metric spaces and demand that $f$ takes only rational values, we still get uncountably many possible distance functions. The remedy is to consider distance functions that are `controlled' by finite sets. Let $A$ be a finite rational metric space. A rational distance function $f:A\times\Nat\rightarrow \Rat^+_0$, i.e. $1$-Lipschitz on $A\times D_K$, is called \emph{finitely-controlled} if there exists a finite subset $N\subseteq \Nat$ such that for any $a\in A$ and $n\in\Nat$ we have $f(a,n)=\max\{0,\max\{f(a,m)-d_K(q_m,q_n):m\in N\}\}$. Clearly, for any rational metric space $A$ there are only countably many finitely-
controlled distance functions.

We shall consider a countable class of finite rational metric structures $\Age_3$ such that $A$ belongs to $\Age_3$ if
\begin{itemize}
\item $A$ is a finite rational metric space,
\item $A$ is equipped with a rational $1$-Lipschitz finitely-controlled distance function $f:A\times\Nat\rightarrow \Rat^+_0$.

\end{itemize}
$\Age_3$ is countable. Joint embedding for two structures $A,B\in \Age_3$ can be achieved by putting $A$ and $B$ far apart from each other. Amalgamation of two structures $B,C\in \Age_3$ over a third structure $A\in \Age_3$ can be again achieved by taking the greatest metric amalgam of $B$ and $C$ over $A$. $f$ is correctly defined on the amalgam since it is $1$-Lipschitz; that follows from the same argument as in Theorems \ref{thm1} and \ref{thm2}.

Thus we get a Fra\" iss\' e limit $(\Rat\Ur, \Rat F)$, where $\Rat F:\Rat\Ur\times\Nat\rightarrow \Rat^+_0$ is the limit of the distance function. By $\Rat C\subseteq \Rat\Ur\times K$ we shall denote the set $\{(u,q_n):\Rat F(u,n)=0\}$. We have the following universality and rational one-point extension property of $(\Rat\Ur, \Rat F)$:
\begin{itemize}
\item for every $A\in \Age_3$ there exists an embedding $\iota:A\hookrightarrow (\Rat\Ur,\Rat F)$, i.e. for every $a\in A$ and $n\in \Nat$ we have $f(a,n)=\Rat F(\iota(a),n)$,
\item for every $A\in \Age_3$, every embedding $\iota:A\hookrightarrow (\Rat\Ur,\Rat F)$, and every one-point extension $A\subseteq B\in \Age_3$ there exists an extension $\iota\subseteq \tilde{\iota}:B\hookrightarrow (\Rat\Ur,\Rat F)$.

\end{itemize}
We take the completion of $(\Rat \Ur,\Rat F)$ to get $(\Ur, F)$, where $F$ is the unique extension of $\Rat F$ onto $\Ur\times \Nat$. Also, by $C\subseteq \Ur\times K$ we denote the closure of $\Rat C$ in $\Ur\times K$.
By $\overline{\Age_3}$ we shall denote the class of finite metric spaces $M$ with $1$-Lipschitz distance function $f:M\times\Nat\rightarrow \Rea^+_0$ that not only can take values in reals, it does not have to be finitely-controlled.

We say that a subclass $K\subseteq \overline{\Age_3}$ has an almost-one-point extension property if for every $A\in K$, every $\varepsilon>0$, and every one-point extension $A\subseteq B=A\coprod\{b\}\in \overline{\Age_3}$ there exists a one-point extension $B'=A\coprod\{b'\}\in K$ such that for every $a\in A$ and every $n\in \Nat$ we have $|d(a,b)-d(a,b')|<\varepsilon$ and $|f(b,n)-f(b',n)|<\varepsilon$. Analogously, we define an almost-one-point extension property for a substructure $(U, F)\subseteq (\Ur, F)$.

As before we need to prove:
\begin{lem}
$(\Ur, F)$ has the one-point extension property.
\end{lem}
That will be again achieved by proving the following three claims.
\begin{claim}\label{thm3claim1}
$\Age_3$ has an almost-one-point extension property.
\end{claim}
\begin{claim}\label{thm3claim2}
If a (not necessarily complete) substructure $(U,F\upharpoonright U\times \Nat)\subseteq (\Ur,F)$ has an almost-one-point extension property, then so does its completion.
\end{claim}
\begin{claim}\label{thm3claim3}
If a complete substructure $(U,F\upharpoonright U\times\Nat)\subseteq (\Ur,F)$ has an almost-one-point extension property, then it has the one-point extension property.
\end{claim}
Claims \ref{thm3claim2} and \ref{thm3claim3} are proved similarly as the analogous Claims \ref{thm1claim2} and \ref{thm1claim3}. We only prove Claim \ref{thm3claim1}.

Fix $A\in \Age_3$, $\varepsilon>0$ and a one-point extension $B=A\coprod\{b\}\in \overline{\Age_3}$. We define the rational metric on $B'$ precisely as we did in the proof of the analogous Claim \ref{thm1claim1}. It remains to define the finitely controlled rational distance function $f$ on $b'$. Note again that we have that for every $a\in A$, $d(a,b)<d(a,b')$, and let $\delta=\min\{\min\{d(a,b')-d(a,b):a\in A\},\varepsilon/3\}$. Since $K$ is compact there exists a finite set $N\subseteq \Nat$ such that $\{q_n:n\in N\}$ is a $\varepsilon/3$-net in $K$. We define a finitely-controlled $f$ controlled by values on $N$ (we may without loss of generality assume that $f$ on $A\times\Nat$ was controlled by values on $N$). For each $n\in N$ we set $f(b',n)$ to be an arbitrary rational number from the interval $[f(b,n),f(b,n)+\delta]$. For other $m\in \Nat$ it extends uniquely by the formula $f(b',m)=\max\{0,\max\{f(b',n)-d_K(q_n,q_m):n\in N\}\}$. We check that $f$ is still $1$-Lipschitz, i.e. for any $a\in A$ and 
$n,
m$ we have $|f(a,n)-f(b',m)|\leq d(a,b')+d_K(q_n,q_m)$. Since $f$ is controlled by values on $N$ it suffices to consider $n,m$ from $N$. We have $$|f(a,n)-f(b',m)|\leq |f(a,n)-f(b,m)|+\delta\leq$$ $$d(a,b)+d_K(q_n,q_m)+\delta\leq d(a,b')+d_K(q_n,q_m).$$
Finally, we need to check that for any $n\in \Nat$ we have $|f(b,n)-f(b',n)|\leq \varepsilon$. It is clear for $n\in N$, so let $n\in \Nat$ be arbitrary. However, since $\{q_n:n\in N\}$ forms an $\varepsilon/3$-net in $K$ there exists $m\in N$ such that $d_K(q_n,q_m)<\varepsilon/3$. Then we have $$|f(b,n)-f(b',n)|\leq |f(b,n)-f(b,m)|+|f(b,m)-f(b',m)|+$$ $$|f(b',m)-f(b',n)|\leq \varepsilon/3+\varepsilon/3+\varepsilon/3$$ since $f$ is $1$-Lipschitz. This finishes the proof.
\end{proof}
\begin{thm}\label{thm4}
Let $Z$ be an arbitrary Polish metric space and $L>0$ an arbitrary real constant. Then there exists an $L$-Lipschitz function $F:\Ur\rightarrow Z$ that is universal and ultrahomogeneous.
\end{thm}
\begin{remark}
$(\Ur,F)$ is uniquely characterized by the following property: let $A_1,A_2$ be two subsets of $\Ur$ such that there is an isometry $\iota:A_1\rightarrow A_2$ with the property that for any $x\in A_1$ we have $F(x)=F\circ \iota(x)$. Then $\iota$ extends to an autoisometry of $\Ur$ with the same property.

As a consequence we get that for any Polish metric space $X$ and an $L$-Lipschitz function $f:X\rightarrow Z$ there is an isometric embedding $\iota:X\hookrightarrow \Ur$ such that for any $x\in X$ we have $f(x)=F\circ \iota(x)$.
\end{remark}
\begin{proof}
Let $D_Z=\{z_n:n\in \Nat\}\subseteq Z$ be some countable dense subset of $Z$. We consider the following countable class $\Age_4$ of finite structures. We have that $A\in \Age_4$ if
\begin{itemize}
\item $A$ is a finite rational metric space
\item there is an $L$-Lipschitz function $p:A\rightarrow D_Z$

\end{itemize}
$\Age_4$ is clearly countable. Joint embedding for two structures $A,B\in \Age_4$ can be again obtained by putting $A$ and $B$ `far apart' from each other; i.e. if $m=\max\{L\cdot d_Z(p(a),p(b):a,b\in A\coprod B\}$ and $M$ is some rational greater than $m$ and $\mathrm{diam}(A)$ and $\mathrm{diam}(B)$, then we can take the disjoint union $A\coprod B$ and define the distance between any $a\in A$ and $b\in B$ to be $2M$. This works. We show the amalgamation property. Suppose we have structures $A,B,C\in \Age_4$, where we assume that $A$ is a common substructure of both $B$ and $C$. As usual, we take the greatest metric amalgam of $B$ and $C$ over $A$ and show that it works. We need to show that for $b\in B$ and $c\in C$ we have $d_Z(p(b),p(c))\leq L\cdot d(b,c)$. Let $a\in A$ be such that $d(b,c)=d(b,a)+d(a,c)$. Then we have $$d_Z(p(b),p(c))\leq d_Z(p(b),p(a))+d_Z(p(a),p(c))\leq$$ $$L\cdot (d(b,a)+d(a,c))=L\cdot d(b,c)$$ and we are done.

We now consider the Fra\" iss\' e limit. As before, it is easy to check that it is the rational Urysohn space $\Rat\Ur$ together with an $L$-Lipschitz function $\Rat P:\Rat\Ur \rightarrow D_Z$. The following universal property and rational one-point extension property characterize $(\Rat\Ur,\Rat P)$:
\begin{itemize}
\item for every $A\in \Age_4$ there exists an embedding $\iota:A\hookrightarrow (\Rat \Ur,\Rat P)$, i.e. $\iota$ is isometric and for every $a\in A$ we have $p(a)=\Rat P\circ \iota(a)$,
\item for every $A\in \Age_4$, every embedding $\iota:A\hookrightarrow (\Rat\Ur,\Rat P)$, and every one-point extension $A\subseteq B\in \Age_4$ there exists an extension $\iota\subseteq \tilde{\iota}:B\hookrightarrow (\Rat\Ur,\Rat P)$.

\end{itemize}
We take the completion $(\Ur,P)$, where $P:\Ur\rightarrow Z$ is the unique extension of $\Rat P$ from $\Rat \Ur$ onto $\Ur$. To finish the proof, we must show the following.
\begin{lem}\label{thmlem4}
$(\Ur,P)$ has a one-point extension property, i.e. for every finite subset $A\subseteq \Ur$ and for every abstract one-point extension $(B=A\coprod \{b\}, \bar{p}:B\rightarrow Z)$, where $\bar{p}=P\upharpoonright A\cup \{(b,z)$ for $z\in Z$, there is a corresponding `concrete' one-point extension in $(\Ur,P)$.
\end{lem}
By $\overline{\Age_4}$ we mean the class of all finite metric spaces equipped with an $L$-Lipschitz function $p$ with values in $Z$, not necessarily in the countable dense set $D_Z$.

We say that a class $K\subseteq \overline{\Age_4}$ has an almost-one-point extension property if for every $A\in K$, every extension $A\subseteq B=A\coprod \{b\}\in \overline{\Age_4}$ and every $\varepsilon>0$, there exists a one-point extension $B'=A\coprod \{b'\}\in K$ such that for every $a\in A$ we have $|d(a,b)-d(a,b')|<\varepsilon$ and $d_Z(p(b),p(b'))<L\cdot \varepsilon$. The almost-one-point extension property for a substructure $(U,P\upharpoonright U)\subseteq (\Ur,P)$ is defined analogously as in the proofs of Theorems \ref{thm1}, \ref{thm2} and \ref{thm3}. Again note that an almost-one-point extension property for $\Age_4$ is equivalent with an almost-one-point extension property for $(\Rat\Ur,\Rat P)$.

As before, we prove Lemma \ref{thmlem4} through the following series of claims.
\begin{claim}\label{thm4claim1}
$\Age_4$ has the almost-one-point extension property.
\end{claim}
\begin{claim}\label{thm4claim2}
If a (not necessarily complete) substructure $(U,P\upharpoonright U)\subseteq (\Ur,P)$ has an almost-one-point extension property, then so does its completion.
\end{claim}
\begin{claim}\label{thm4claim3}
If a complete substructure $(U,P\upharpoonright U)\subseteq (\Ur,P)$ has an almost-one-point extension property, then it has the one-point extension property.
\end{claim}
We shall only prove Claim \ref{thm4claim1}, the proofs of the other two claims are routine and modifications of the analogous ones from the proofs of Theorems \ref{thm1}, \ref{thm2} and \ref{thm3}.

Let $A\in \Age_4$ and $A\subseteq B\in \overline{\Age_4}$ with some $\varepsilon>0$ be given, where $B=A\coprod \{b\}$. We define $A\subseteq B'\in \Age_4$ as follows. $B'=A\coprod \{b'\}$ and for every $a\in A$ we define $d(a,b')$ as in Claim \ref{thm1claim1}. In particular, we again have that 
\begin{equation}\label{thm4eq1}
d(a,b')> d(a,b)
\end{equation}
for every $a\in A$. If $p(b)\in D_Z$, then we set $p(b')=p(b)$ and because of \eqref{thm4eq1} we have $d_Z(p(a)-p(b'))=d_Z(p(a),p(b))\leq L\cdot d(a,b)\leq d(a,b')$ for every $a\in A$ and we are done. If $p(b)\notin D_Z$, in particular $p(b)$ is not an isolated point of $Z$, we choose some $z_0\in D_Z$ such that $d_Z(z_0,p(b))<L\cdot \delta$, where $\delta=\min \{d(a,b')-d(a,b):a\in A\}>0$. Then we put $p(b')=z_0$ and we claim that this works. Indeed, for any $a\in A$ we have $$d_Z(p(a),z_0)\leq d_Z(p(a),p(b))+L\cdot \delta\leq L\cdot (d(a,b)+\delta)\leq L\cdot d(a,b').$$
This finishes the proof of Theorem \ref{thm4}.
\end{proof}
We have an interesting corollary of Theorem \ref{thm4}. This theorem `lifts' from the category of complete metric spaces to the category of Banach spaces via the functor assigning to a metric space its Lipschitz-free Banach space and to a Lipschitz map its linear extension. We refer the reader to the book \cite{Wea} for information about Lipschitz-free Banach spaces and to the paper \cite{GoKa} of Godefroy and Kalton. Here we just recall that for every metric space $X$ (with a distinguished point representing $0$) there exists a Banach space $F(X)$ in which there is an isometric copy of $X$ such that $\Span\{X\}$ is dense in $F(X)$ and that is uniquely characterized by the property that for every Banach space $Y$ and Lipschitz map $f:X\rightarrow Y$ sending $0$ to $0$ there exists a unique bounded linear operator with the same Lipschitz constant $\hat{f}:F(X)\rightarrow Y$ that extends $f$.

In case $X$ is a Banach space and $\iota:X\rightarrow F(X)$ is the canonical isometric embedding, then by $\beta:F(X)\rightarrow X$ we denote $\hat{\iota^{-1}}$, i.e. the unique linear operator from $F(X)$ to $X$ such that we have $\beta\circ\iota=\mathrm{id}$.

The following theorem was proved in \cite{GoKa}.
\begin{thm}[Godefroy,Kalton \cite{GoKa}]\label{Kalton}
For any separable Banach space $X$ there exists a linear isometry $\iota_{GK}:X\rightarrow F(X)$ such that $\beta\circ \iota_{GK}=\mathrm{id}_X$.
\end{thm}
It follows that the so-called Holmes space $\Hol$, the Lipschitz-free space $F(\Ur)$ over the Urysohn space, is a universal separable Banach space. We refer the reader to the paper of Holmes \cite{Holm} and to Chapter 5 of \cite{Pe} for more information about this Banach space. We remark that it was proved by Fonf and Wojtaszczyk in \cite{Woj} that the Holmes universal space is not linearly isometric to other known universal Banach spaces such as the Gurarij space or $C([0,1])$. It is also not isomorphic to the Pe\l czy\' nski universal space which also follows from the results from \cite{Woj}.

We note that universal and homogeneous linear operators on the Gurarij space and $p$-Gurarij spaces were constructed in \cite{GarKub} and \cite{CGK}.
\begin{thm}\label{thm4cor}
Let $Z$ be an arbitrary separable Banach space and $L>0$ an arbitrary real constant. Then there exists a universal linear operator $\Phi:\Hol\rightarrow Z$ of norm $L$.
\end{thm}
\begin{proof}
Let $F:\Ur\rightarrow Z$ be the universal $L$-Lipschitz map from the Urysohn space to the Banach space $Z$. Denote by $\Phi$ the unique linear extension of $F$ from $\Ur$ to $\Hol$, where we chose $0$ in $\Ur$ so that $F(0)=0$. We claim that $\Phi$ is as desired.

Indeed, let $X$ be an arbitrary separable Banach space equipped with a linear operator $\psi:X\rightarrow Z$ such that $\|\psi\|\leq L$. Using Theorem \ref{thm4} we obtain a (non-linear) isometric embedding $\iota:X\hookrightarrow \Ur$ such that for every $x\in X$ we have 
\begin{equation}\label{eq1}
F\circ \iota(x)=\psi(x).
\end{equation}

We shall again denote by $X'$ the image of $X$ in $\Ur$ and by $\beta:F(X')\subseteq \Hol\rightarrow X$ the canonical surjective linear operator from $F(X')$ onto $X$ so that we have $\beta\circ\iota=\mathrm{id}_X$. By Theorem \ref{Kalton} there exists a linear isometry $\iota_{GK}:X\rightarrow F(X')\subseteq \Hol$ such that $\beta\circ\iota_{GK}=\mathrm{id}_X$. We claim that for every $x\in X$ we have $\Phi\circ \iota_{GK}=\psi$. Once we prove it we are done.

Consider the linear operator $\Phi '=\psi\circ \beta$. We have $$\Phi '\circ \iota_{GK}=\psi\circ \beta\circ \iota_{GK}=\psi\circ\mathrm{id}_X=\psi$$ thus it suffices to prove that $\Phi\upharpoonright F(X') =\Phi '$. However, $\Phi\upharpoonright F(X')$ is uniquely determined by the property that for every $x\in X'$ we have $\Phi(x)=F(x)$. But if we take any $x\in X'$ then $$\Phi '(x)=\psi\circ\beta(x)=\psi\circ\iota^{-1}(x)=F(x)$$ where the last equality follows from \eqref{eq1} and we are done.
\end{proof}
\section{Gurarij space}
Our aim in this section is to prove similar universality and homogeneity results that we did for the Urysohn space for the Gurarij space. We recall the definition of the Gurarij space.
\begin{defin}\label{defGur}
Recall that a separable Banach space $\Gur$ is Gurarij if it satisfies the following property: for every $\varepsilon>0$, every finite dimensional Banach spaces $E\subseteq F$, and every linear isometry $\phi:E\hookrightarrow \Gur$ there exists an extension $\bar{\phi}\supseteq \phi:F\hookrightarrow \Gur$ such that $\bar{\phi}$ is an `$\varepsilon$-isometry', i.e. for every $x\in F$ we have $(1-\varepsilon)\cdot \|x\|\leq \|\bar{\phi}(x)\|\leq (1+\varepsilon)\cdot \|x\|$.
\end{defin}

Before formulating the theorems, let us define a necessary notion. Let $(S_E,E)$ and $(S_F,F)$ be pairs of $1$-Lipschitz seminorms together with Banach spaces. Let $\varepsilon>0$. Then we say that $\phi:(S_E,E)\rightarrow (S_F,F)$ is an \emph{$\varepsilon$-morphism} if $\phi$ is a linear $\varepsilon$-isometry between $E$ and $F$ in their norms and also an $\varepsilon$-isometry between the quotients of $E$ and $F$ by their respective seminorms, i.e. for every $x\in E$ we have $|S_E(x)-S_F(x)|<\varepsilon\cdot S_E(x)$.

Since we shall speak always about seminorms that are $1$-Lipschitz, we may sometimes omit the adjective `$1$-Lipschitz' when talking about them.

The following simple fact that we state without a proof shows why we consider $1$-Lipschitz seminorms.
\begin{fact}
Let $X$ be a Banach space. There is a (not always one-to-one) correspondence between closed subspaces $Y\subseteq X$ and $1$-Lipschitz seminorms $S:X\rightarrow \Rea^+_0$.

Namely, the function $\dist(\cdot,Y):X\rightarrow \Rea^+_0$ is a $1$-Lipschitz seminorm and the set $\{x\in X:S(x)=0\}$ is a closed subspace.
\end{fact}
We shall use the following notation: whenever $X$ is a Banach space and $Y\subseteq X$ is a closed subspace, then by $S_X^Y:X\rightarrow \Rea^+_0$ we denote the seminorm $\dist(\cdot,Y)$. Conversely, if $S:X\rightarrow \Rea^+_0$ is a $1$-Lipschitz seminorm, then by $Y_S\subseteq X$ we denote the closed subspace $\{x\in X: S(x)=0\}$.

The following theorems are the main results. The first one is analogous to Theorem \ref{thm1}, we only replace $1$-Lipschitz function by $1$-Lipschitz seminorms. The second one is analogous to Theorem \ref{thm2}, we only replace $1$-Lipschitz retraction by $1$-Lipschitz projection. As in the previous section, we formulate the theorems rather informally and provide precise formulations in the remarks below.
\begin{thm}\label{Guthm1}
There exists a universal and almost-homogeneous closed subspace of the Gurarij space, resp. a $1$-Lipschitz seminorm on the Gurarij space.
\end{thm}
\begin{remark}
Precisely, we claim that there exists a $1$-Lipschitz seminorm $\mathbb{S}:\Gur\rightarrow \Rea^+_0$ on the Gurarij space such that the pair $(\mathbb{S},\Gur)$ has the following homogeneity property: for every finite dimensional Banach space $E$ with a seminorm $S_E: E\rightarrow \Rea^+_0$ and for every finite dimensional extension $F$ equipped with a seminorm $S_F:F\rightarrow \Rea^+_0$ that extends $S_E$, for every $\varepsilon>0$ and for every $0$-morphism $\phi:(S_E,E)\hookrightarrow (\mathbb{S},\Gur)$, there exists an extension $\bar{\phi}\supseteq \phi:(S_F,F)\hookrightarrow (\mathbb{S},\Gur)$ that is an $\varepsilon$-morphism.

Consequently, for every separable Banach space $X$ with a closed subspace $Y$ there exists a linear isometry $\phi:X\hookrightarrow \Gur$ such that for every $x\in X$ we have $\dist(x,Y)=\dist(\phi(x),\mathbb{H})$, where $\mathbb{H}=\{x\in \Gur:\mathbb{S}(x)=0\}$ may be understood as a \emph{universal closed subspace} of the Gurarij space.
\end{remark}
If $X$ is a Banach space, $P:X\rightarrow Y$ is a norm one projection onto its closed subspace and $S:X\rightarrow \Rea^+_0$ is a seminorm, then we say that $P$ and $S$ are \emph{compatible} if $S\leq S_X^Y$ (i.e. for every $x\in X$, $S(x)\leq S_X^Y(x)$), and for every $x\in X$, $S(x)=0$ iff $S_X^Y(x)=0$, i.e. $\{x\in X:S(x)=0\}=Y$.
\begin{thm}\label{Guthm2}
There exists a universal and almost-homogeneous norm one projection $P_E:\Gur\rightarrow \mathbb{H}$ onto a $1$-complemented subspace.
\end{thm}
\begin{remark}
Precisely, we claim that there exists a norm one projection $P:\Gur\rightarrow \mathbb{H}$ onto a $1$-complement subspace of the Gurarij space with the following homogeneity property: for every finite dimensional Banach space $E$ together with a norm one projection $P_E:E\rightarrow E_0$ onto its $1$-complemented subspace $E_0$ and a compatible seminorm $S_X$,  and for every finite dimensional extension $F$ equipped with a norm one projection $P_F\supseteq P_E:F\rightarrow F_0$ onto its $1$-complemented subspace $F_0$ and a compatible seminorm $S_F\supseteq S_E$, for every $\varepsilon>0$ and for every $0$-morphism $\phi:(S_E,E)\hookrightarrow (S_\Gur^{\mathbb{H}},\Gur)$ such that for every $x\in E$ we have $\phi\circ P_E(x)=P\circ\phi(x)$ there exists an extension $\bar{\phi}\supseteq \phi:(S_F,F)\hookrightarrow (S_\Gur^{\mathbb{H}},\Gur)$ such that $\bar{\phi}$ is an $\varepsilon$-morphism and again for every $x\in F$ we have $\bar{\phi}\circ P_F(x)=P\circ\bar{\phi}(x)$.

Consequently, for every separable Banach space $X$ with a norm one projection $p:X\rightarrow X_0$ there exists a linear isometry $\phi:X\hookrightarrow \Gur$ such that for every $x\in X$ we have $\dist(x,X_0)=\dist(\phi(x),\mathbb{H})$, and in particular, $P\circ\phi(x)=\phi\circ p(x)$.
\end{remark}
\begin{remark}
The operator $P$ from Theorem \ref{Guthm2} is thus a universal and homogeneous projection. As pointed out to us by the referee, every operator $T:X\rightarrow Y$ (of norm $1$) can be lifted to a (norm $1$) projection. Indeed, just consider the sum $X\oplus_1 Y$ and define a projection $\bar{T}:X\oplus_1 Y\rightarrow Y$ by $\bar{T}(x,y):=T(x)+y$. It follows that every universal projection is actually a universal operator. Both the projection $P$ from Theorem \ref{Guthm2} and the universal operators from \cite{GarKub} and \cite{CGK} are characterized uniquely with respect to their universality properties, which are however formally different. It is thus left open whether these operators are the same.
\end{remark}
Before we prove these two theorems we present a construction of the Gurarij space that is `Fra\" iss\' e like' in the classical sense. Then we will be able to prove Theorems \ref{Guthm1}, resp. \ref{Guthm2} using minor modifications of proofs of Theorems \ref{thm1}, resp. \ref{thm2}.
\begin{defin}\label{partialnorm}
Let $X$ be a vector space and let $A\subseteq X$ be a subset such that $\Span\{A\}=X$. A partial $A$-norm $\norm_A$ is a non-negative real function which behaves like a norm except that it is defined only on $A$; i.e. for every $x,y\in A$ and $\alpha\in \Rea$ we have
\begin{itemize}
\item $\|x\|_A=0$ iff $x=0$,
\item $\|\alpha\cdot x\|_A=|\alpha|\cdot \|x\|_A$ if $\alpha\cdot x\in A$,
\item $\|x+y\|_A\leq \|x\|_A+\|y\|_A$ if $x+y\in A$.

\end{itemize}
\end{defin}
\begin{fact}\label{partialnormext}
Let $X$ be a vector space, $A\subseteq X$ a subset such that $\Span\{A\}=X$ and let $\norm_A:A\rightarrow \Rea$ be a partial $A$-norm. Then for any $x\in X$ the formula  $$\|x\|^A_X=\inf\{|\alpha_1|\cdot \|a_1\|_A+\ldots+|\alpha_n|\cdot\|a_n\|_A:$$ $$a_1,\ldots,a_n\in A,\alpha_1\cdot a_1+\ldots+\alpha_n\cdot a_n=x\}$$ defines a maximal seminorm $\norm^A_X$ on $X$ that extends $\norm_A$.

If $X$ is finite dimensional and $A$ is finite, then $\norm^A_X$ is actually a norm and the infimum in the formula may be replaced by minimum.
\end{fact}
\begin{proof}
It is easy to check that it is a seminorm that extends $\norm_A$. Since every seminorm $\norm$ extending $\norm_A$ must satisfy the inequality $\|x\|\leq |\alpha_1|\cdot \|a_1\|_A+\ldots+|\alpha_n|\cdot\|a_n\|_A$ for every $a_1,\ldots,a_n\in A$ such that $\alpha_1\cdot a_1+\ldots+\alpha_n\cdot a_n=x$, we have that $\norm^A_X$ is maximal.

Now suppose that $X$ is finite dimensional and that $A=\{a_1,\ldots,a_n\}$. Fix some $x\in X$ and take some $\alpha_1,\ldots,\alpha_n\in \Rea$ such that $\alpha_1\cdot a_1+\ldots+\alpha_n\cdot a_n=x$. Let $\delta=|\alpha_1|\cdot \|a_1\|_A+\ldots+|\alpha_n|\cdot\|a_n\|_A$, thus $\|x\|^A_X\leq\delta$. The following is a compact subset of $\Rea^n$: $K_x=\{(\beta_1,\ldots,\beta_n):\beta_1\cdot a_1+\ldots+\beta_n\cdot a_n=x,|\beta_1|\cdot \|a_1\|_A+\ldots+|\beta_n|\cdot\|a_n\|_A\leq \delta\}$. Moreover, the map $(\beta_1,\ldots,\beta_n)\to |\beta_1|\cdot \|a_1\|_A+\ldots+|\beta_n|\cdot\|a_n\|_A$ is continuous, thus attains the minimum at some tuple $(\beta_1,\ldots,\beta_n)\in K_x$. It also follows that $\norm^A_X$ is a norm.
\end{proof}
\begin{defin}
Let $\Age$ be the following class of Banach spaces. We have that $X\in \Age$ if:
\begin{itemize}
\item $X$ is a finite dimensional vector space with a specified basis $(x_1,\ldots,x_n)$,
\item the norm $\norm$ on $X$ is of the form $\norm^A_X$, i.e. determined by a partial norm $\norm_A$. Moreover, we demand that $A$ is a finite subset of $X$ containing the basis such that each element of $A$ is a linear combination of elements of the basis using only rational scalars, and $\norm_A:A\rightarrow \Rat$ is a partial norm taking values only in the rationals.

\end{itemize}
The class of embeddings $\mathcal{E}$ consists only of those linear isometric embeddings between elements of $\Age$ that send elements of basis to elements of basis; i.e. if $X,Y\in \Age$, where the basis of $X$ is $(x_1,\ldots,x_n)$ and the basis of $Y$ is $(y_1,\ldots,y_m)$, then an allowed linear isometric embedding from $X$ into $Y$ is determined by an injection $\iota:\{1,\ldots,n\}\rightarrow \{1,\ldots,m\}$. We shall call such linear embeddings \emph{proper}.
\end{defin}
\begin{remark}
Although the definition of the class $\Age$ above is more suitable for our purposes, we note that it is nothing else but the class of finite dimensional \emph{rational} Banach spaces as considered e.g. in \cite{GarKub2}. We recall that a finite dimensional Banach space $X$ is called rational if it is isometric to $(\Rea^n,\norm)$ such that the unit sphere is a polytope whose vertices have rational coordinates. However, given $X\in\Age$ whose norm is determined by some finite $A$ we see that the unit sphere is the convex hull of $\{a/\|a\|:a\in A\}$, where clearly each $a/\|a\|$ has rational coordinates in the specified basis.
\end{remark}
\begin{fact}\label{GurFraisse}
$\Age$ is a Fra\" iss\' e class.
\end{fact}
\begin{proof}
It immediately follows that $\Age$ is countable.

Let us check the amalgamation property. Let $X_0,X_1,X_2\in \Age$, where $X_0$ is a common subspace of $X_1$ and $X_2$. Moreover, we have that $(x_1,\ldots,x_n)$ is a basis of $X_0$, $(x_1,\ldots,x_n,y_1,\ldots,y_m)$ is a basis of $X_1$ and $(x_1,\ldots,x_m,z_1,\ldots,z_k)$ is a basis of $X_2$. The amalgam space $X_3$ is algebraically nothing else but the amalgamated direct sum $X_1\oplus_{X_0} X_2$ of $X_1$ and $X_2$ over $X_0$. The norm is defined again in a standard way, i.e the amalgamation norm (analogous to the greatest metric amalgamation) $\|x-y\|=\inf\{\|x-z\|_{X_1}+\|z-y\|_{X_2}:z\in X_0\}$ for $x\in X_1$, $y\in X_2$. To check that $X_3\in \Age$ first observe that it is a vector space with basis $(x_1,\ldots,x_n,y_1,\ldots,y_m,z_1,\ldots,z_k)$. If $\norm_{X_1}$ was given by some partial norm $\norm_{A_1}$ for finite $A_1\subseteq X_1$, and $\norm_{X_2}$ was given by some partial norm $\norm_{A_2}$ for finite $A_2\subseteq X_2$, then considering $A_1,A_2$ as subsets of $X_3$ we can form a partial 
norm $\norm_{A_3}$ defined on $A_3=A_1\cup A_2$ so that for $a\in A_3$, $\|a\|_{A_3}$ is equal to $\|a\|_{A_1}$ if $a\in A_1$ and equal to $\|a\|_{A_2}$ if $a\in A_2$. It is straightforward to check that this defines a partial norm on $A_3$ and that the extension $\norm^{A_3}_{X_3}$ is the amalgam norm on $X_3$. For a reader unfamiliar with amalgam metrics, resp. norms, we refer to our paper \cite{Do} where a similar fact was verified for norms (resp. invariant metrics) on abelian groups.

The joint embedding property is similar, only easier.
\end{proof}
Thus there exists the Fra\" iss\' e limit $G$, a direct limit of some countable sequence $X_1\to X_2\to\ldots$ from $\Age$. The following extension property follows from the Fra\" iss\' e theorem.
\begin{fact}\label{extGur}
Let $Y\in \Age$ and let there be a proper linear embedding $\phi:X_n\hookrightarrow Y$ for some $n$. Then there exist $m>n$ and a proper linear embedding $\psi:Y\hookrightarrow X_m$ such that $\psi\circ\phi=\subseteq _{n\to m}$, where $\subseteq_{n\to m}$ is the inclusion proper embedding from $X_n$ into $X_m$.
\end{fact}
It is a separable normed space and we take the completion which we shall denote by $\Gur$.
\begin{thm}
$\Gur$ is the Gurarij space.
\end{thm}
\begin{proof}
We need to check the condition from Definition \ref{defGur}. It is sufficient to prove the following:
\begin{claim}\label{Gurclaim}
For every $\varepsilon>\varepsilon'> 0$, every finite dimensional Banach spaces $E\subseteq F$, where 
\begin{itemize}
\item $E$ is of co-dimension $1$ in $F$, $F=\Span\{E,v\}$ and $v\in F\setminus E$ such that $\|v\|=1$,
\item $\varepsilon '<\min\{\varepsilon/4,\frac{\varepsilon\cdot\delta}{14}\}$, where  $\delta=\mathrm{dist}(v,E)=\inf\{\|v-x\|:x\in E\}$,

\end{itemize}
we have that any $\varepsilon '$-isometry $\phi :E\hookrightarrow \Gur$ extends to $\bar{\phi}\supseteq \phi:F\hookrightarrow \Gur$ such that $\bar{\phi}$ is an $\varepsilon$-isometry.
\end{claim}
Indeed, suppose that we have proved Claim \ref{Gurclaim} and we are given subspaces $E\subseteq F$, $\varepsilon>0$ and an isometry $\phi:E\hookrightarrow \Gur$. Suppose the co-dimension of $E$ in $F$ is $n$ and $E$ has basis $\{e_1,\ldots,e_m\}$, which can be extended to a basis $\{e_1,\ldots,e_m,f_1,\ldots,f_n\}$ of $F$. Then the extension of $\phi$ to $\bar{\phi}$ is done using Claim \ref{Gurclaim} $n$-times through spaces $E=E_0\leq\ldots\leq E_i=\Span\{E,f_1,\ldots,f_i\}\leq\ldots E_n=F$ so that the extension $\phi_i\supseteq \ldots\supseteq\phi: E_i\hookrightarrow \Gur$ is an $\varepsilon_i$-isometry, where $\varepsilon_i<\min\{\varepsilon_{i+1}/4,\frac{\varepsilon_{i+1}\cdot \delta_i}{14}\}$, where $\delta_i=\mathrm{dist}(E_i,f_{i+1})$.\\

Thus we need to prove Claim \ref{Gurclaim}. Suppose that $E$ is a subspace of co-dimension $1$ in a finite dimensional Banach space $F\supseteq E$. Let $(e_1,\ldots,e_n)$ be a basis of $E$ and $(e_1,\ldots,e_n,v)$ a basis in $F$ such that $\|v\|=1$. Let $\varepsilon>\varepsilon '>0$ be given, where $\varepsilon '<\min\{\varepsilon/4,\frac{\varepsilon\cdot\delta}{14}\}$,  $\delta=\mathrm{dist}(v,E)$. Moreover suppose we have an $\varepsilon '$-isometry $\phi:E\hookrightarrow \Gur$.

For any $\gamma>0$ and $i\leq n$ we can find $g_i\in G\subseteq \Gur$ such that $\|g_i-\phi(e_i)\|<\gamma$. Moreover, there exists some $m$ such that $g_i\in X_m\subseteq G$ for all $i$ (recall again that $G$ is a direct limit of some sequence $(X_i)_i$ from $\Age$). We may even assume that each $g_i$ is a rational linear combination of elements of the specified basis of $X_m$. It is clear that if we take $\gamma>0$ small enough then the linear map $\phi':E\hookrightarrow X_m$ determined by sending $e_i$ to $g_i$ is an $\varepsilon ''$-isometry for some $\varepsilon ''<\min\{\varepsilon/4,\frac{\varepsilon\cdot\delta}{14}\}$.

Take now $R>0$ large enough to be specified later, and $\varepsilon '>\eta>0$, and find some finite $\eta$-net $(z_j)_j$ in $R\cdot B_E$, the compact ball of radius $R$ in $E$. We may assume that each $z_j$ from the net is a rational linear combination of the basis elements $\{e_1,\ldots,e_n\}$. Consider now an (abstract) extension $Z$ of $X_m$ generated by $X_m$ and one additional vector $w$ of norm $1$. The norm on $X_m$ (which is a restriction of the norm on $\Gur$ to this subspace) is an extension of some rational partial norm $\norm_A$, where $A$ is a finite set of rational linear combinations of the basis of $X_m$. Extend $A$ to $\bar{A}$ so that it contains $\phi'(z_j)$ for every $j$ and $w$. Note that each $z_j$ is a rational linear combination of basis elements $\{e_1,\ldots,e_n\}$ and $\phi'(e_i)$ is a linear combination of basis elements of $X_m$, thus $\phi'(z_j)$ is also a rational linear combination of basis elements in $X_m$. Extend the partial rational norm $\norm_A$ on $A$ to a partial 
rational norm $\norm_{\bar{A}}$ on $\bar{A}$ so that for every $j$ we have 
\begin{equation}\label{Gurclaimeq}
| \|z_j-v\|-\|\phi'(z_j)-w\|_{\bar{A}}|<\varepsilon ''.
\end{equation}
That can be done as in the proof of Claim \ref{thm1claim1}. We consider the norm $\norm_Z^{\bar{A}}$ on $Z$ that extends the partial norm $\norm_{\bar{A}}$; it coincides with the $X_m$-norm on the subspace $X_m$. By Fact \ref{extGur} this `abstract' extension $Z$ is realized in $G\subseteq \Gur$, thus we may suppose that actually $w$ is an element of $\Gur$.

We now claim that the extension $\bar{\phi}\supseteq \phi:F\hookrightarrow \Gur$ determined by sending $v$ to $w$ is as desired. It suffices to check that for any $y\in E$ we have $$|\|v-y\|-\|w-\phi(y)\| |<\varepsilon\cdot\|v-y\|.$$ Suppose at first that $y\in R\cdot B_E$. Then we pick some $z_j$ from the $\eta$-net such that $\|y-z_j\|<\eta$. Then we have $$| \|v-y\|-\|w-\phi(y)\| |\leq $$ $$| \|v-z_j\|-\|w-\phi '(z_j)\|+\|y-z_j\|+\|\phi(y)-\phi(z_j)\|+\|\phi(z_j)-\phi '(z_j)\| |<$$ 
$$\varepsilon ''+\eta + \varepsilon '\cdot \eta +(\varepsilon'+\varepsilon'')\cdot\|z_j\|,$$
where we used that $| \|v-z_j\|-\|w-\phi '(z_j)\| |<\varepsilon''$ by \eqref{Gurclaimeq}, $\|y-z_j\|<\eta$ by the choice of $z_j$, $\|\phi(y)-\phi(z_j)\|<\varepsilon'\cdot \eta$ since $\phi$ has norm less than $\varepsilon'$, and $\|\phi(z_j)-\phi '(z_j)\|<(\varepsilon'+\varepsilon'')\cdot \|z_j\|$ since $\phi'$ has norm less than $\varepsilon''$.\\

Then we get
$$\varepsilon ''+\eta + \varepsilon '\cdot \eta +(\varepsilon'+\varepsilon'')\cdot\|z_j\|<\varepsilon ''+\eta + \varepsilon '\cdot \eta +(\varepsilon'+\varepsilon'')\cdot(\eta+\|y\|)<$$ $$\varepsilon ''+\eta + \varepsilon '\cdot \eta +2\varepsilon ''\cdot(\eta+1+\|v-y\|)< \varepsilon ''+\varepsilon ''+\varepsilon ''+2\varepsilon ''+2\varepsilon ''+2\varepsilon''\cdot \|v-y\|<$$ $$\frac{\varepsilon\cdot \mathrm{dist}(v,E)}{2}+\frac{\varepsilon\cdot \|v-y\|}{2}<\varepsilon\cdot\|v-y\|,$$
where in the last inequality we used that $7\varepsilon''<7\frac{\dist(v,E)}{14}=\frac{\dist(v,E)}{2}$ and that $2\varepsilon''\cdot\|v-y\|<2\varepsilon/4\cdot\|v-y\|$.\\
 
Suppose now that $y\notin R\cdot B_E$. Then we have $$| \|v-y\|-\|w-\phi(y)\| |\leq $$ $$\|v\|+\|w\|+\|y-\phi(y)\| \leq 2+\varepsilon '\cdot \|y\|\leq 3+\varepsilon '\cdot \|v-y\|\leq$$ $$\frac{3}{R-1}\cdot \|v-y\|+\varepsilon '\cdot\|v-y\|=(\frac{3}{R-1}+\varepsilon ')\cdot\|v-y\|.$$ Clearly, if $R$ is large enough, $\frac{3}{R-1}+\varepsilon '$ is less than $\varepsilon$.
\end{proof}
Knowing the construction of the Gurarij space which is in the similar vein as the construction of the Urysohn space we will be rather easily able to transfer the results about universal structures on the Urysohn space to analogous results on the Gurarij space, i.e. Theorems \ref{Guthm1} and \ref{Guthm2}. The proofs will be sketchy as it is a repetition of very similar arguments to those from the section on the Urysohn space.\\

\begin{proof}[Proof of Theorem \ref{Guthm1}.]
We shall define an appropriate Fra\" iss\' e class. Before doing so, analogously as in Definition \ref{partialnorm} and Fact \ref{partialnormext} we define a partial seminorm and show how to extend it.

Let $X$ be a normed space and $A\subseteq X$ subset such that $\Span\{A\}=X$. Let $S_A:A\rightarrow \Rea^+_0$ be a partial $1$-Lipschitz seminorm.

Then we can consider as in Fact \ref{partialnormext} the greatest extension of $S_A$ to $S_X:X\rightarrow \Rea^+_0$ as follows: for any $x\in X$ we set $$S_X(x)=\inf\{|\alpha_1|\cdot S_A(x_1)+\ldots+|\alpha_n|\cdot S_A(x_n):$$ $$x_1,\ldots,x_n\in A, x=\alpha_1\cdot x_1+\ldots+\alpha_n\cdot x_n,\alpha_1\geq 0,\ldots,\alpha_n\geq 0\}.$$

\begin{defin}
$\Age_1$ will be the following class of pairs of seminorms and Banach spaces. We have that $(S_X,X)\in \Age_1$ if:
\begin{itemize}
\item $X$ is a finite dimensional vector space with a specified basis $(x_1,\ldots,x_n)$,
\item the norm $\norm$ on $X$ is of the form $\norm^A_X$, i.e. determined by a partial norm $\norm_A$. We again demand that $A$ is a finite subset of $X$ containing the basis such that each element of $A$ is a linear combination of elements of the basis using only rational scalars, and $\norm_A:A\rightarrow \Rat$ is a partial norm taking values only in the rationals. Moreover, we shall assume that each basis element $x_i$, $i\leq n$, is of norm $1$.
\item $S_X$ is the (greatest) extension of some $1$-Lipschitz seminorm $S_A:A\rightarrow \Rat^+_0$ on $A$ taking only rational values.

\end{itemize}
We shall again consider only proper linear embeddings, i.e. those linear isometric embeddings between elements of $\Age_1$ that send elements of the basis to elements of basis, that are moreover $0$-morphisms, i.e. preserve the seminorms.
\end{defin}
The verification that $\Age_1$ is a Fra\" iss\' e class is essentially the same as in Fact \ref{GurFraisse} plus some some arguments from the proof of Theorem \ref{thm1} on the universal subset of the Urysohn space. The Fra\" iss\' e limit $(S,G)$ is a direct limit of some pairs $(S_1,X_1)\to (S_2,X_2)\to\ldots$. Let $H\subseteq G$ be the linear subspace $\{x\in G:S(x)=0\}$. We denote by $\Gur$ the completion of $G$, by $\mathbb{S}$ the unique extension of $S$ and by $\mathbb{H}$ the completion of the subspace $H$ in $\Gur$. We need to check that the pair $(\mathbb{S},\Gur)$ satisfies the condition from the statement of Theorem \ref{Guthm1}. It will be again sufficient to prove the following claim.
\begin{claim}\label{Gurclaim2}
For every $\varepsilon>\varepsilon'> 0$, every pairs $(S_E,E)\subseteq (S_F,F)$, where 
\begin{itemize}
\item $E,F$ are finite dimensional Banach spaces and $E$ is of co-dimension $1$ in $F$, $F=\Span\{E,v\}$ and $v\in F\setminus E$ such that $\|v\|=1$,
\item $S_E,S_F$ are seminorms where $S_F$ extends $S_E$,
\item $\varepsilon '<\min\{\varepsilon/2,\frac{\varepsilon\cdot\delta}{10}\}$, where  $\delta=\mathrm{dist}(v,E)=\inf\{\|v-x\|:x\in E\}$.

\end{itemize}
we have that any $\varepsilon '$-morphism $\phi :(S_E,E)\hookrightarrow (\mathbb{S},\Gur)$ extends to $\bar{\phi}\supseteq \phi:(S_F,F)\hookrightarrow (\mathbb{S},\Gur)$ such that $\bar{\phi}$ is an $\varepsilon$-morphism.
\end{claim}
That is essentially the same as the proof of Claim \ref{Gurclaim}. We again define the mapping $\phi '$ that goes from $E$ to some $X_m\subseteq G\subseteq \Gur$ and sends basis elements of $E$ to rational linear combinations of basis elements of $X_m$. Then we again find an appropriate $\eta$-net in a large enough ball $R\cdot B_E$. When defining the abstract extension $Z=\Span\{X_m,w\}$ of $X_m$, the only difference is that besides the norm we also have to define the seminorm $S$ on this extension. We do it precisely the same as we did for the norm.\\

We show how the `homogeneity condition' from the statement of the theorem implies the universality. We do it analogously as in the proof of universality of the Gurarij space in \cite{KuSo}.

We need the following Lemma which is analogous to Lemma 2.1 in \cite{KuSo}.
\begin{lem}\label{Gurlemma}
Let $(S_X,X)$ and $(S_Y,Y)$ be pairs of finite dimensional Banach spaces together with seminorms, and let $\phi:(S_X,X)\rightarrow (S_Y,Y)$ be an $\varepsilon$-morphism for some $\varepsilon\geq 0$. Then there exist a pair $(S_W,W)$, where $W$ is finite dimensional, and $0$-morphisms $\iota_X:(S_X,X)\rightarrow (S_W,W)$ and $\iota_Y:(S_Y,Y)\rightarrow (S_W,W)$ such that $\|\iota_Y\circ \phi-\iota_X\|<\varepsilon$.
\end{lem}
Suppose for a moment that the lemma has been proved. Then the rest is done like the proof of the universality of the Gurarij space in \cite{KuSo} with $\varepsilon$-isometries replaced by $\varepsilon$-morphisms.
Let $X$ be a separable Banach space together with a $1$-Lipschitz seminorm $S_X:X\rightarrow \Rea^+_0$. Let $(X_n)_n$ be an increasing chain of finite dimensional subspaces of $X$ such that $\overline{\bigcup_n X_n}=X$. Denote by $S_n$ the restriction $S_X\upharpoonright X_n$. We inductively find linear embeddings $\phi_n: X_n\hookrightarrow \Gur$ so that
\begin{itemize}
\item $\phi_n:(S_n,X_n)\rightarrow (\mathbb{S},\Gur)$ is a $1/2^n$-morphism,
\item $\|\phi_{n+1}\upharpoonright X_n-\phi_n\|<1/2^{n-1}$.

\end{itemize}
Once this is done we take the point-wise limit $\phi:\bigcup_n X_n\hookrightarrow \Gur$. It uniquely extends to a linear isometric embedding still denoted by $\phi$ on $X$ with the property that for each $x\in X$ we have $|S_X(x)-\mathbb{S}(\phi(x))|<\varepsilon$, for every $\varepsilon>0$ thus $S_X(x)=\mathbb{S}(\phi(x))$. In particular, if $S_X$ is a distance function from a closed subspace $Y\subseteq X$, then for each $x\in X$ we have $x\in Y$ iff $\phi(x)\in \mathbb{H}$.

Let us now find such linear embeddings $\phi_n$'s. We assume that $X_1=\{0\}$. Suppose we have found a $1/2^n$-morphism $\phi_n:(S_n,X_n)\rightarrow (\mathbb{S},\Gur)$. Denote by $X'_n\subseteq \Gur$ the image $\phi_n(X_n)$, and by $\phi'_n:X'_n\rightarrow X_{n+1}$ the inverse $\phi_n^{-1}$ (composed with the inclusion $\subseteq X_n\rightarrow X_{n+1}$). $\phi'_n$ is also a $1/2^n$-morphism. Using Lemma \ref{Gurlemma} we can find a pair $(S_W,W)$ and $0$-morphisms $\iota_n:(\mathbb{S}\upharpoonright X'_n,X'_n)\rightarrow (S_W,W)$ and $\iota_{n+1}:(S_{n+1},X_{n+1})\rightarrow (S_W,W)$ such that $\|\iota_{n+1}\circ \phi'_n-\iota_n\|<1/2^n$. Then using the homogeneity property of $(\mathbb{S},\Gur)$ we can find a $1/2^{n+1}$-morphism $\psi:(S_W,W)\rightarrow (\mathbb{S},\Gur)$ such that $\psi\circ \phi_n=\mathrm{id}_{X'_n}$. The desired $1/2^{n+1}$-morphism $\phi_{n+1}:(S_{n+1},X_{n+1})\rightarrow (\mathbb{S},\Gur)$ is then the composition $\psi\circ \iota_{n+1}$.

It remains to prove Lemma \ref{Gurlemma}.
\begin{proof}[Proof of Lemma \ref{Gurlemma}.]
We refer the reader to the proof of Lemma 2.1 in \cite{KuSo} which is formulated precisely the same as Lemma \ref{Gurlemma}, just with $\varepsilon$-isometries instead of $\varepsilon$-morphisms.

Since being an $\varepsilon$-morphism means being an $\varepsilon$-isometry in both the norm and the seminorm we just use Lemma 2.1 from \cite{KuSo} twice. Going through the proof of that lemma we see that $W$ is $X\oplus Y$ with a suitable norm $\|\cdot\|'$ and $\iota_X$ and $\iota_Y$ are the canonical embeddings. We then apply Lemma 2.1 from \cite{KuSo} again for the quotient spaces $X_Q$ and $Y_Q$ (quotiented by their respective seminorms) to obtain a suitable norm $\|\cdot\|''$ on $X_Q\oplus Y_Q$. The desired seminorm $S_W$ is then the composition of projection from $X\oplus Y$ to $X_Q\oplus Y_Q$ with the norm $\|\cdot\|''$.
\end{proof}
\end{proof}
\begin{proof}[Proof of Theorem \ref{Guthm2}.]
Let us start by defining the appropriate Fra\" iss\' e class.
\begin{defin}
$\Age_2$ will be the class of triples of seminorms, projections and Banach spaces. We have that $(S_X,p_X,X)\in \Age_2$ if:
\begin{itemize}
\item $X$ is a finite dimensional vector space with a specified basis $(x_1,\ldots,x_n)$,
\item the norm $\norm$ on $X$ is of the form $\norm^A_X$, i.e. determined by a partial norm $\norm_A$. We again demand that $A$ is a finite subset of $X$ containing the basis such that each element of $A$ is a linear combination of elements of the basis using only rational scalars, and $\norm_A:A\rightarrow \Rat$ is a partial norm taking values only in the rationals. Moreover, we shall assume that each basis element $x_i$, $i\leq n$, is of norm $1$.
\item $S_X$ is again the (greatest) extension of some $1$-Lipschitz seminorm $S_A:A\rightarrow \Rat^+_0$ on $A$ taking only rational values.
\item $X$ is equipped with a norm one projection $p$ that is allowed to send basis elements only to rational linear combinations of basis elements, and moreover $Y_{S_X}=\{x\in X: S_X(x)=0\}=p_X(X)=\{x\in X:p_X(x)=x\}$; i.e. $S_X$ and $p_X$ are compatible.

\end{itemize}
We shall again consider only proper linear embeddings, i.e. those linear isometric embeddings between elements of $\Age_2$ that send elements of the basis to elements of basis, that are moreover $0$-morphisms and commute with projections. So an allowed embedding between $(S_X,p_X,X)$ and $(S_Y,p_Y,Y)$ is a $0$-morphism $\phi$ that sends specified basis elements of $X$ to specified basis elements of $Y$, and for every $x\in X$, $p_Y\circ\phi(x)=\phi\circ p_X(x)$.
\end{defin}
The verification that it is indeed a Fra\" iss\' e class is again based on the same arguments as in the proofs of Fact \ref{GurFraisse} and Theorem \ref{thm2}. The Fra\" iss\' e limit $G$ is again a direct limit of some spaces $X_1\to X_2\to\ldots$ that are equipped with the seminorms $s_1,s_2,\ldots$ and projections $p_1,p_2,\ldots$ that also extend to the limit and then to the completion $\Gur$. We shall denote this limit projection $P$ and its range by $\mathbb{H}$. The limit seminorm is equal to the seminorm $S_\Gur^{\mathbb{H}}$. In order to check the condition from the statement of Theorem \ref{Guthm2} it is again sufficient to prove the following claim.
\begin{claim}
For every $\varepsilon>\varepsilon'> 0$, every finite dimensional Banach spaces $E\subseteq F$, where 
\begin{itemize}
\item $E$ is equipped with a seminorm $S_E$ and a compatible norm one projection $p_E$ that both extend to $F$,
\item $E$ is of co-dimension $1$ in $F$, $F=\Span\{E,v\}$ and $v\in F\setminus E$ such that $\|v\|=1$,
\item $\varepsilon '<\min\{\varepsilon/2,\frac{\varepsilon\cdot\delta}{10}\}$, where  $\delta=\mathrm{dist}(v,E)=\inf\{\|v-x\|:x\in E\}$,

\end{itemize}
we have that any $\varepsilon '$-morphism $\phi :(S_E,E)\hookrightarrow (S_\Gur^{\mathbb{H}},\Gur)$ with the property that for every $x\in E$ we have $P\circ \phi(x)=\phi\circ p(x)$, extends to $\bar{\phi}\supseteq \phi:(S_F,F)\hookrightarrow (S_\Gur^{\mathbb{H}},\Gur)$ such that $\bar{\phi}$ is an $\varepsilon$-morphism with the analogous property.
\end{claim}
That is again essentially the same as the proofs of Claim \ref{Gurclaim} and then Claim \ref{Gurclaim2}. We define the mapping $\phi '$ that goes from $E$ to some $X_m$ and sends basis elements of $E$ to rational linear combinations of basis elements of $X_m$. Then we again find an appropriate $\eta$-net in a large enough ball $R\cdot B_E$ which we may suppose contains $\phi'(e_i)$ for every $i\leq n=\mathrm{dim}(E)$. When defining the abstract extension $Z=\Span\{X_m,w\}$ of $X_m$, we only additionally specify to where $w$ projects. It is analogous as in the proof of Claim \ref{thm2claim1}.

The universality and uniqueness of $(\Gur,P)$ is then again a standard argument. Let us only mention that when proving the universality in the same way as in Theorem \ref{Guthm1} or in paper \cite{KuSo}, we need the following lemma which is an analog of Lemma \ref{Gurlemma}, resp. Lemma 2.1 from \cite{KuSo}.
\begin{lem}
Let $(S_X,p_X,X)$ and $(S_Y,p_Y,Y)$ be triples consisting of finite dimensional Banach spaces equipped with seminorms and compatible norm one projections. Let $\phi:(S_X,X)\rightarrow (S_Y,Y)$ be an $\varepsilon$-morphism, for some $\varepsilon>0$, with the property that  $p_Y\circ\phi=\phi\circ p_X$. Then there exist a finite dimensional $W$ with a seminorm $S_W$ and a compatible norm one projection $p_W$, and $0$-morphisms $\iota_X:(S_X,X)\rightarrow (S_W,W)$, resp. $\iota_Y:(S_Y,Y)\rightarrow (S_W,W)$, with the property that for $Q\in\{X,Y\}$ we have $$p_W\circ \iota_Q=\iota_Q\circ p_Q$$ and moreover we have $$\|\iota_Y\circ \phi-\iota_X\|<\varepsilon.$$
\end{lem}
We just copy the proof of Lemma \ref{Gurlemma}, where $W=X\oplus Y$. We only additionally define a projection $p_W$ on $W=X\oplus Y$ which is the sum $p_X\oplus p_Y:X\oplus Y\rightarrow p_X(X)\oplus p_Y(Y)$.
\end{proof}
\subsection{Final remarks and problems}
Consider the Urysohn space together with a universal closed subset $C\subseteq \Ur$ as guaranteed by Theorem \ref{thm1}. Can this universal closed subset be lifted to a universal subspace of $\Hol$? Resp. is $F(C)\subseteq \Hol$ a universal subspace? Using just Theorem \ref{Kalton} as in the proof of Theorem \ref{thm4cor} does not seem to work. Maybe a modification of Theorem \ref{Kalton} is needed.

The same question applies to the universal retraction on $\Ur$. Is the unique linear extension $\hat{R}:\Hol\rightarrow F(F_\Ur)$ of the universal retraction $R$ a universal projection on the Holmes onto a universal complemented subspace $F(F_\Ur)$? Again, the approach from the proof of Theorem \ref{thm4cor} that uses Theorem \ref{Kalton} does not seem to work directly.

\end{document}